\documentclass[11pt]{article}
\usepackage{mathrsfs}
\usepackage{amsmath}
\usepackage{amsfonts}
\usepackage{amsmath,amsthm}
\usepackage{amsmath,amssymb,amsthm,latexsym}
\usepackage{amscd}
\usepackage[all]{xy}
\usepackage{hyperref}

\newtheorem{theorem}{Theorem}[section]

\newtheorem{definition}[theorem]{Definition}

\newtheorem{lemma}[theorem]{Lemma}

\theoremstyle{remark}
\newtheorem{remark}[theorem]{Remark}
\def\o{\omega}
\def\<{\langle}
\def\>{\rangle}

\begin{document}
\title{\bf{Wintgen ideal submanifolds of codimension two, complex curves, and M\"obius geometry}}
\author {Tongzhu Li, Xiang Ma, Changping Wang, Zhenxiao Xie}
\maketitle

\begin{abstract}
Wintgen ideal submanifolds in space forms are those ones
attaining the equality pointwise in the so-called DDVV inequality which relates the scalar curvature, the mean curvature and the scalar normal curvature. Using the framework of M\"obius geometry, we show that in the codimension two case ($M^m\to\mathbb{S}^{m+2}$), the mean curvature sphere of the Wintgen ideal submanifold corresponds to an 1-isotropic holomorphic curve in a complex quadric $Q^{m+2}_+$.
Conversely, any 1-isotropic complex curve in $Q^{m+2}_+$ describes a 2-parameter family of $m$-dimensional spheres whose envelope is always a $m$-dimensional Wintgen ideal submanifold at the regular points. The relationship with Dajczer and Tojeiro's work on the same topic as well as the description in terms of minimal surfaces in the Euclidean space is also discussed.
\end{abstract}

\hspace{2mm}


{\bf MSC(2000):\hspace{2mm} 53A10, 53C42, 53C45}

\section{Introduction}

A remarkable result in the submanifold theory in real space forms is the so-called DDVV inequality, which relates the most important intrinsic and extrinsic quantities at an arbitrary point of a submanifold (like the scalar curvature, the mean curvature, and the scalar normal curvature), without any restriction on the dimension/codimension or any further geometric/topological assumptions. This universal inequality was a difficult conjecture in \cite{DDVV1, DDVV}, and was finally proved in \cite{Ge} and \cite{Lu3}.

It is very interesting to characterize the equality case in the DDVV inequality. By the suggestion of \cite{chen10,ml} and the characterization of \cite{Ge} about the equality case at an arbitrary point, we introduce the following definition.

\begin{definition}
A submanifold $M^m$ of dimension $m$ and codimension $p$ in a real space form is called a \emph{Wintgen ideal submanifold} if the equality is attained at every point of $M^m$. This happens if, and only if, at every point $x\in M$ there exists an orthonormal basis $\{e_1,\cdots,e_m\}$ of the tangent plane $T_xM^m$ and an orthonormal basis $\{n_1,\cdots,n_p\}$ of the normal plane $T_x^{\bot}M^m$,
such that the shape operators $\{A_{n_i},i=1,\cdots,p\}$ take the form as below \cite{Ge}:
\begin{equation}\label{form1}
A_{n_1}=
\begin{pmatrix}
\lambda_1 & \mu_0 & 0 & \cdots & 0\\
\mu_0 & \lambda_1 & 0 & \cdots & 0\\
0  & 0 & \lambda_1 & \cdots & 0\\
\vdots & \vdots & \vdots & \ddots & \vdots\\
0  & 0 & 0 & \cdots & \lambda_1
\end{pmatrix},
A_{n_2}=
\begin{pmatrix}
\lambda_2+\mu_0 & 0 & 0 & \cdots & 0\\
0 & \lambda_2-\mu_0 & 0 & \cdots & 0\\
0  & 0 & \lambda_2 & \cdots & 0\\
\vdots & \vdots & \vdots & \ddots & \vdots\\
0  & 0 & 0 & \cdots & \lambda_2
\end{pmatrix},
\end{equation}
$$A_{n_3}=\lambda_3I_p,~~~~ A_{n_r}=0, r\ge 4.$$
Note that the distribution $\mathbb{D}=Span\{e_1,e_2\}$ is well-defined when the submanifold is umbilic-free. This is called \emph{the canonical distribution}.
\end{definition}

Wintgen ideal submanifolds are abundant. Wintgen \cite{wint} first proved the DDVV inequality for surfaces in $\mathbb{S}^4$. When the equality is attained everywhere, such surfaces are called super-conformal, which means that the curvature ellipse is a circle, or equivalently, the Hopf differential is an isotropic differential form. For more examples see \cite{br,Dajczer1,Dajczer2,Dajczer3,Lu,XLMW}. Generally a Wintgen ideal submanifold is not necessarily minimal; on the other hand, it is noteworthy that many important examples appearing in the partial classification results above come from holomorphic curves or minimal surfaces/submanifolds.

An important observation by Dajczer and Tojeiro \cite{Dajczer3} (based on the result in \cite{DDVV}) is that, the DDVV inequality, as well as the equality case, are invariant under M\"{o}bius transformations of the ambient space. So it is clear that the most suitable framework for the study of Wintgen ideal submanifolds is M\"{o}bius geometry.

This research program has been carried out by us in \cite{LiTZ2} and \cite{XLMW}. In \cite{LiTZ2}, we show that when the canonical distribution $\mathbb{D}$ generates a comparatively lower dimensional integrable distribution, a Wintgen ideal submanifold is a cylinder, a cone, or a rotational submanifold over a minimal Wintgen ideal submanifold in $\mathbb{R}^n, S^n$ or $H^n$, respectively. In \cite{XLMW}, when the dimension $m=3$ and the codimension $p=2$, we show that $M^3$ has a circle bundle structure over a Riemann surface among other results. Observe that the sphere bundle structure manifests itself in both cases.

In this paper we concentrate on the codimension two case and consider an important M\"obius invariant object associated with a submanifold $M$, the so-called \emph{mean curvature sphere}.
At each point $x\in M^m$, it is the unique $m$-dimensional round sphere tangent to $M^m$ at $x$ which also shares the same mean curvature vector with $M^m$ at $x$. In the codimension two case,
this assigns a (oriented) space-like 2-space $\mathrm{Span}_{\mathbb{R}}\{\xi_1,\xi_2\}$ in the Lorentz space $\mathbb{L}^{m+4}$, which is also identified with
the isotropic complex line $\mathrm{Span}_{\mathbb{C}}\{\xi_{1}-i\xi_{2}\}\in\mathbb{C}P^{m+3}$
(with respect to the $\mathbb{C}$-linear extension of the Lorentz metric). When the base point $x$ varies along $M^m$, we obtain the mean curvature sphere congruence, which is also represented as a Gauss map
\[
[\xi]\triangleq[\xi_{1}-i\xi_{2}]: M^m\to Q^{m+2}_+=\{[Z]\in \mathbb{C}P^{m+3}| \langle Z,Z\rangle =0,~\langle Z,\bar{Z} \rangle>0\}.
\]
It is similar to the generalized Gauss map of a minimal surface in Euclidean space and the conformal Gauss map of a (Willmore) surface \cite{br0}.

The first key observation by us (see also \cite{XLMW}) is that under the hypothesis of being Wintgen ideal, this $m$-sphere congruence is indeed a 2-parameter family, and its envelope not only recovers $M^{m}$, but also extends it to a submanifold as a sphere bundle over a Riemann surface $\overline{M}$ (a holomorphic curve).
The underlying surface $\overline{M}$ comes from
the quotient surface $\overline{M}=M^{m}/\Gamma$ (at least locally) where $\Gamma$ is the foliation of $M^{m}$ by the integral submanifolds of the distribution $\mathbb{D}^{\bot}=\mathrm{Span}\{e_3,\cdots,e_m\}$. Moreover,
the mean curvature sphere $[\xi_1-i\xi_2]$ indeed determines a
holomorphic, 1-isotropic curve in $Q^{m+2}_+$, and all codimension two Wintgen ideal submanifolds can be constructed by such curves in $Q^{m+2}_+$. The precise statement of our main result is as below.

\begin{theorem}
The mean curvature spheres $[\xi]\triangleq[\xi_{1}-i\xi_{2}] \in Q^{m+2}_+$ of a Wintgen ideal submanifold of codimension two is a holomorphic and 1-isotropic curve, i.e.,
\[
\xi_{\bar{z}}~\parallel~\xi, ~~\<\xi_z,\xi_z\>=0.
\]
Conversely, given a holomorphic isotropic curve
\[
[\xi]:\overline{M}\to Q^{m+2}_+\subset\mathbb{C}P^{m+3},
\]
the envelope $\widehat{M}^m$ of the corresponding 2-parameter family spheres is a
$m$-dimensional Wintgen ideal submanifold (at the regular points).
\end{theorem}

This paper is organized as below. In Section~2 we give a brief review of the submanifold theory in M\"obius geometry. In Section~3 we restrict to consider Wintgen ideal submanifold of codimension two. The M\"obius invariants now take a much simpler expression.

As the core of this paper, in Section~4 we show that the mean curvature sphere congruence of a codimension two Wintgen ideal submanifold defines a 1-isotropic holomorphic curve in $Q^{m+2}_+$, and in Section~5 the converse is also proved.
The geometry of such curves in $Q^{m+2}_+$ is also briefly explained in Section~5. These two parts finish the proof to the main theorem mentioned above.

It should be noted that Dajczer and Tojeiro already gave another description of Wintgen ideal submanifolds of codimension two in \cite{Dajczer3} via minimal surfaces in $\mathbb{R}^{m+2}$.
Their construction is compared with ours in Section~6.
Indeed, these two descriptions are equivalent by a correspondence between holomorphic, 1-isotropic curves in $Q^{m+2}_+\subset\mathbb{C}P^{m+3}_1$ (the \emph{Gauss map} of $M^m$) and those ones in $\mathbb{C}^{m+2}$ (the generalized Gauss map of minimal surfaces in $\mathbb{R}^{m+2}$). This comes from a correspondence between $Q^{m+2}$ and $\mathbb{C}^{m+2}$ which could be regarded as a complex version of the classical stereographic projection.

In particular, we give a new proof to the following fact: When the ambient space is endowed with the Euclidean flat metric, the centers of the mean curvature spheres of a codimension two Wintgen ideal submanifold constitutes a minimal surface in this Euclidean space. This beautiful result was first obtained by Rouxel \cite{Rouxel} for superconformal surfaces (i.e., Wintgen ideal surfaces) in $\mathbb{R}^4$, then re-discovered and generalized by Dajczer and Tojeiro \cite{Dajczer2, Dajczer3} for arbitrary codimensional case.\\

\textbf{Acknowledgement} This work is funded by the Project 10901006 and 11171004 of National Natural Science Foundation of China.

\section{Basic invariants and equations for submanifolds in M\"obius geometry}

In this section we briefly review the theory of submanifolds
in M\"obius geometry. For details we refer to $\cite{CPWang}$.

In the classical light-cone model, the light-like directions in the Lorentz space $\mathbb{R}^{m+p+2}_1$ correspond to points in the round sphere $\mathbb{S}^{m+p}$, and the Lorentz orthogonal group correspond to conformal transformation group of $\mathbb{S}^{m+p}$. The Lorentz inner product between
$Y=(Y_0,Y_1,\cdots,Y_{m+p+1}), Z=(Z_0,Z_1,\cdots,Z_{m+p+1})\in
\mathbb{R}^{m+p+2}_1$ is
\[
\langle  Y,Z\rangle=-Y_0Z_0+Y_1Z_1+\cdots+Y_{m+p+1}Z_{m+p+1}.
\]

Let $x:M^m\rightarrow \mathbb{S}^{m+p}\subset \mathbb{R}^{m+p+1}$ be a submanifold without umbilics. Take $\{e_i|1\le i\le m\}$ as the tangent frame with respect to the induced metric $I=dx\cdot dx$, and $\{\theta_i\}$ as the dual 1-forms.
Let $\{n_{r}|1\le r\le p\}$ be orthonormal frame for the
normal bundle. The second fundamental form and
the mean curvature of $x$ are
\begin{equation}\label{2.1}
II=\sum_{ij,\gamma}h^{r}_{ij}\theta_i\otimes\theta_j
n_{r},~~H=\frac{1}{m}\sum_{j,r}h^{r}_{jj}n_{r}=\sum_{r}H^{r}n_{r},
\end{equation}
respectively. We define the M\"{o}bius position vector $Y:
M^m\rightarrow \mathbb{R}^{m+p+2}_1$ of $x$ by
\begin{equation}\label{2.2}
Y=\rho(1,x),~~~
~~\rho^2=\frac{m}{m-1}\left|II-\frac{1}{m} tr(II)I\right|^2
\end{equation}
which is also called the canonical lift of $x$ \cite{CPWang}.
Two submanifolds $x,\bar{x}: M^m\rightarrow \mathbb{S}^{m+p}$
are M\"{o}bius equivalent if there exists $T$ in the Lorentz group
$O(m+p+1,1)$ in $\mathbb{R}^{m+p+2}_1$ such that $\bar{Y}=YT.$
It follows immediately that
\begin{equation}
\mathrm{g}=\langle dY,dY\rangle=\rho^2 dx\cdot dx
\end{equation}
is a M\"{o}bius invariant, called the M\"{o}bius metric of $x$.

Let $\Delta$ be the Laplacian with respect to $\mathrm{g}$. Define
\begin{equation}
N=-\frac{1}{m}\Delta Y-\frac{1}{2m^2}
\langle \Delta Y,\Delta Y\rangle Y,
\end{equation}
which satisfies
\[
\langle Y,Y\rangle=0=\langle N,N\rangle, ~~
\langle N,Y\rangle=1~.
\]
Let $\{E_1,\cdots,E_m\}$ be a local orthonormal frame for $(M^m,\mathrm{g})$
with dual 1-forms $\{\omega_1,\cdots,\omega_m\}$. Write
$Y_j=E_j(Y)$. Then we have
\[
\langle  Y_j,Y\rangle =\langle  Y_j,N\rangle =0, ~\langle  Y_j,Y_k\rangle =\delta_{jk}, ~~1\leq j,k\leq m.
\]
We define
\[
\xi_r=(H^r,n_r+H^r x).
\]
Then $\{\xi_{1},\cdots,\xi_p\}$ form the orthonormal frame of the
orthogonal complement of $\mathrm{Span}\{Y,N,Y_j|1\le j\le m\}$.
And $\{Y,N,Y_j,\xi_{r}\}$ form a moving frame in $\mathbb{R}^{m+p+2}_1$ along $M^m$.
\begin{remark}\label{rem-xi}
Geometrically, at one point $x$, $\xi_r$ (for any given $r$) corresponds to the unique hyper-sphere tangent to $M_m$ with normal vector $n_r$ and mean curvature $H^r(x)$. In particular, the spacelike subspace $\mathrm{Span}_{\mathbb{R}}\{\xi_1,\cdots,\xi_p\}$ represents a unique $m$-dimensional sphere tangent to $M_m$ with the same mean curvature vector $\sum_r H^r n_r$. This well-defined object was naturally named \emph{the mean curvature sphere} of $M^m$ at $x$. Note that it still share the same mean curvature at $x$ even when the ambient space is endowed with any other conformal metric.
\end{remark}
We fix the range of indices in this section as below: $1\leq
i,j,k\leq m; 1\leq r,s\leq p$. The structure equations are:
\begin{equation}\label{eq-structure}
\begin{split}
&dY=\sum_i \omega_i Y_i,\\
&dN=\sum_{ij}A_{ij}\omega_i Y_j+\sum_{i,r} C^r_i\omega_i \xi_{r},\\
&d Y_i=-\sum_j A_{ij}\omega_j Y-\omega_i N+\sum_j\omega_{ij}Y_j
+\sum_{j,r} B^{r}_{ij}\omega_j \xi_{r},\\
&d \xi_{r}=-\sum_i C^{r}_i\omega_i Y-\sum_{i,j}\omega_i
B^{r}_{ij}Y_j +\sum_{s} \theta_{rs}\xi_{s},
\end{split}
\end{equation}
where $\omega_{ij}$ are the connection $1$-forms of the M\"{o}bius
metric $\mathrm{g}$ and $\theta_{rs}$ the normal connection $1$-forms. The
tensors
\begin{equation}
{\bf A}=\sum_{i,j}A_{ij}\omega_i\otimes\omega_j,~~ {\bf
B}=\sum_{i,j,r}B^{r}_{ij}\omega_i\otimes\omega_j \xi_{r},~~
\Phi=\sum_{j,r}C^{r}_j\omega_j \xi_{r}
\end{equation}
are called the Blaschke tensor, the M\"{o}bius second fundamental
form and the M\"{o}bius form of $x$, respectively.
The covariant derivatives $A_{ij,k}, B^{r}_{ij,k}, C^{r}_{i,j}$ are defined as usual. For example,
\begin{eqnarray*}
&&\sum_j C^{r}_{i,j}\omega_j=d C^{r}_i+\sum_j C^{r}_j\omega_{ji}
+\sum_{s} C^{s}_i\theta_{sr},\\
&&\sum_k B^{r}_{ij,k}\omega_k=d B^{r}_{ij}+\sum_k
B^{r}_{ik}\omega_{kj} +\sum_k B^{r}_{kj}\omega_{ki}+\sum_{s}
B^{s}_{ij}\theta_{sr}.
\end{eqnarray*}
The integrability conditions for the structure equations are given as below:
\begin{eqnarray}
&&A_{ij,k}-A_{ik,j}=\sum_{r}B^{r}_{ik}C^{r}_j
-B^{r}_{ij}C^{r}_k,\label{equa1}\\
&&C^{r}_{i,j}-C^{r}_{j,i}=\sum_k(B^{r}_{ik}A_{kj}
-B^{r}_{jk}A_{ki}),\label{equa2}\\
&&B^{r}_{ij,k}-B^{r}_{ik,j}=\delta_{ij}C^{r}_k
-\delta_{ik}C^{r}_j,\label{equa3}\\
&&R_{ijkl}=\sum_{r}B^{r}_{ik}B^{r}_{jl}-B^{r}_{il}B^{r}_{jk}
+\delta_{ik}A_{jl}+\delta_{jl}A_{ik}
-\delta_{il}A_{jk}-\delta_{jk}A_{il},\label{equa4}\\
&&R^{\perp}_{rs ij}=\sum_k
B^{r}_{ik}B^{s}_{kj}-B^{s}_{ik}B^{r}_{kj}. \label{equa5}
\end{eqnarray}
Here $R_{ijkl}$ denote the curvature tensor of $\mathrm{g}$.
Other restrictions on tensor $\bf B$ are
\begin{equation}
\sum_j B^{r}_{jj}=0, ~~~\sum_{i,j,r}(B^{r}_{ij})^2=\frac{m-1}{m}. \label{equa7}
\end{equation}
All coefficients in the structure equations are determined by $\{\mathrm{g}, {\bf B}\}$
and the normal connection $\{\theta_{\alpha\beta}\}$.
Coefficients of M\"{o}bius invariants and
the isometric invariants are related as below. (We omit the formula for $A_{ij}$ since it will not be used later.)
\begin{align}
B^{r}_{ij}&=\rho^{-1}(h^{r}_{ij}-H^{r}\delta_{ij}),\label{2.22}\\
C^{r}_i&=-\rho^{-2}[H^{r}_{,i}+\sum_j(h^{r}_{ij}
-H^{r}\delta_{ij})e_j(\ln\rho)]. \label{2.23}
\end{align}
\begin{remark}
For $x: M^{m} \rightarrow \mathbb{R}^{m+p}$, the M\"obius position vector $Y: M^{m}\rightarrow \mathbb{R}^{m+p+2}_1$ and the mean curvature sphere $\{\xi_{1},\cdots,\xi_p\}$ are given by
\[
Y=\rho(\frac{1+|x|^2}{2}, \frac{1-|x|^2}{2}, x),
\]
\[
\xi_r=\left(\frac{1+|x|^2}{2}, \frac{1-|x|^2}{2}, x\right)H^r+(x\cdot n_r,-x\cdot n_r,n_r).
\]
\end{remark}

\section{Wintgen ideal submanifolds of codimension 2}

From now on, we assume $x: M^{m}\to \mathbb{S}^{m+2}$ to be a codimension two Wintgen ideal submanifold. According to \cite{Dajczer3} and \cite{Ge}, that means we can choose a suitable tangent and normal frame ($\{E_1,\cdots,E_m\}$ and $\{\xi_1,\xi_2\}$) such the M\"obius second fundamental form $B$ can be written down as follows:
\begin{equation}\label{3.1}
B^{1}=
\begin{pmatrix}
0 & \mu & 0 & \cdots & 0\\
\mu & 0 & 0 & \cdots & 0\\
0  & 0 & 0 & \cdots & 0\\
\vdots & \vdots & \vdots & \ddots & \vdots\\
0  & 0 & 0 & \cdots & 0
\end{pmatrix},~~
B^{2}=
\begin{pmatrix}
\mu & 0 & 0 & \cdots & 0\\
0 & -\mu & 0 & \cdots & 0\\
0  & 0 & 0 & \cdots & 0\\
\vdots & \vdots & \vdots & \ddots & \vdots\\
0  & 0 & 0 & \cdots & 0
\end{pmatrix},~~
\mu=\sqrt{\frac{m-1}{4m}}.
\end{equation}
Note that $\mu=\sqrt{\frac{m-1}{4m}}$ is a constant determined by \eqref{equa7}.
The canonical distribution $\mathbb{D}=\mathrm{Span}\{E_1,E_2\}$ is well-defined, as well as its orthogonal distribution $\mathbb{D}^\bot=\mathrm{Span}\{E_3,\cdots,E_m\}$.

For convenience we adopt the convention below on the range of indices:
\[
1\le i,j,k,l\le m,~~3\le a,b\le m.
\]
First we compute the covariant derivatives of $B^r_{ij}$.
The result is
\begin{equation}\label{eq-B1}
\begin{split}
&\sum_{j} B^{1}_{11,j}\omega_j=-\mu\theta, ~~~\sum_{j} B^{2}_{11,j}\omega_j=~0,\\
&\sum_{j} B^{1}_{12,j}\omega_j=~0, ~~~\sum_{j} B^{2}_{12,j}\omega_j=~\mu\theta,\\
&\sum_{j} B^{1}_{22,j}\omega_j=\mu\theta, ~~~\sum_{j} B^{2}_{22,j}\omega_j=~0,
\end{split}
\end{equation}
and
\begin{equation}\label{eq-B2}
\begin{split}
&\sum_{j} B^{1}_{1a,j}\omega_j=\mu\omega_{2a}, ~~~\sum_{j} B^{2}_{1a,j}\omega_j=\mu\omega_{1a},\\
&\sum_{j} B^{1}_{2a,j}\omega_j=\mu\omega_{1a}, ~~~\sum_{j} B^{2}_{2a,j}\omega_j=-\mu\omega_{2a},\\
&~~~~~~~~~~B^{1}_{ab,j}=0, ~~~B^{2}_{ab,j}=0.
\end{split}
\end{equation}
Here
\begin{equation}\label{eq-theta}
\theta\triangleq 2\omega_{12}+\theta_{12}
\end{equation}
is a combination of the connection 1-forms
of the bundle $\mathbb{D}$ and the normal bundle.

Since $B^r_{ij,k}$ is symmetric on distinct $i,j,k$ by \eqref{equa3},
by \eqref{eq-B1} and \eqref{eq-B2},
\begin{equation}
B^{1}_{1a,2}=B^{1}_{2a,1}=B^{1}_{12,a}=0;
~~~B^{1}_{1a,b}=B^{1}_{2a,b}=0 ~(\text{if}~a\ne b).
\end{equation}
Again by \eqref{eq-B2} we obtain
\begin{eqnarray}
\omega_{1a}&=&\frac{1}{\mu}\sum_{j} B^{1}_{2a,j}\omega_j
=\frac{1}{\mu}(B^{1}_{2a,2}\omega_2+B^{1}_{2a,a}\omega_a)
=\frac{1}{\mu}\sum_{j} B^{2}_{1a,j}\omega_j,\label{eq-omega1a}\\
\omega_{2a}&=&\frac{1}{\mu}\sum_{j} B^{1}_{1a,j}\omega_j
=\frac{1}{\mu}(B^{1}_{1a,1}\omega_1+B^{1}_{1a,a}\omega_a)
=-\frac{1}{\mu}\sum_{j} B^{2}_{2a,j}\omega_j.\label{eq-omega2a}
\end{eqnarray}
Comparing the components above implies
\begin{equation}\label{eq-B3}
B^{2}_{1a,1}=B^{2}_{2a,2}=0.
\end{equation}
By \eqref{eq-B2}, $B^{2}_{11,a}=B^{2}_{22,a}=0$. Together with \eqref{equa3}\eqref{eq-B3} we get that for $3\le a\le m$,
\begin{eqnarray}
C^{2}_a&=&B^{2}_{11,a}-B^{2}_{1a,1}=0; \label{eq-C1}\\
C^{1}_a&=&B^{1}_{11,a}-B^{1}_{1a,1}=
-\mu\theta(e_3)-\mu\omega_{23}(e_1) \notag\\
&=&B^{1}_{22,a}-B^{1}_{2a,2}=\mu\theta(e_3)-\mu\omega_{13}(e_2). \label{eq-C2}
\end{eqnarray}
Similarly there is
\begin{equation}\label{eq-C3}
\begin{split}
&C^{1}_1=B^{1}_{22,1}=\mu\theta(e_1)=-B^{1}_{1a,a}
=-\mu\omega_{2a}(e_a), \\
&C^{1}_2=B^{1}_{11,2}=-\mu\theta(e_2)=-B^{1}_{2a,a}
=-\mu\omega_{1a}(e_a), \\
&C^{2}_1=-B^{2}_{12,2}=-\mu\theta(e_2)=-B^{2}_{1a,a}
=-\mu\omega_{1a}(e_a), \\
&C^{2}_2=-B^{2}_{12,1}=-\mu\theta(e_1)=-B^{2}_{2a,a}
=\mu\omega_{2a}(e_a).
\end{split}
\end{equation}
In particular we have
\[C^{1}_1=-C^{2}_2,~~C^{1}_2=C^{2}_1.\]
Also note that the normal connection 1-form $\theta_{12}=-\theta_{21}$. Substitute these relations into the last structure equation in \eqref{eq-structure}, we obtain
\begin{align*}
d\xi_1&=-(C_1^1\omega_1+C_2^1\omega_2)Y
-\mu(\omega_1Y_2+\omega_2Y_1)+\theta_{12}\xi_2,\\
d\xi_2&=-(C_2^1\omega_1-C_1^1\omega_2)Y
-\mu(\omega_1Y_1-\omega_2Y_2)-\theta_{12}\xi_1.
\end{align*}
Combining these two equations and re-writing them using the complexified frame, we obtain an elegant formula as below:
\begin{equation}\label{J1}
d(\xi_{1}-i\xi_{2})=i\mu(\o_1+i\o_2)(\eta_1+i\eta_2)
+i\theta_{12}(\xi_1-i\xi_2),
\end{equation}
where
\begin{equation}\label{eq-eta1}
\eta_1=Y_1+\frac{C_2^1}{\mu}Y,~~
\eta_2=Y_2+\frac{C^1_1}{\mu}Y.
\end{equation}
This formula and its geometric explanation is the focus of this paper.

As a preparation, we point out that for a codimension two submanifold in $\mathbb{S}^{m+2}$, the mean curvature sphere defines a Gauss map into the Grassmann manifold $Gr(2,\mathbb{R}^{m+4}_1)$, the module space of spacelike 2-planes in the lorentz space which is a pseudo-Riemannian symmetric space. This can be identified with a non-compact complex quadric $Q^{m+2}_+\subset \mathbb{C}P^{m+3}$ via the following correspondence
\[
\mathrm{Span}_{\mathbb{R}}\{\xi_1,\xi_2\}~\leftrightarrow~
[\xi_1-i\xi_2]\in\mathbb{C}P^{m+3}.
\]

\section{The geometry of 1-isotropic complex curves in $Q^{m+2}_+$}

To describe $Q^{m+2}_+$, note that the complex space $\mathbb{C}^{m+4}_1=\mathbb{R}^{m+4}_1\otimes\mathbb{C}$ is endowed with the complex inner product coming from the bilinear extension of the Lorentz metric.
The null lines in this space form a $m+2$ dimensional compact complex quadric hypersurface
\[
Q^{m+2}=\{[\xi]\in \mathbb{C}P^{m+3}|~
\xi\in\mathbb{C}^{m+4}_1,\<\xi,\xi\>=0\}.
\]
This $\xi$ is either a complex multiple of a light-like vector in $\mathbb{R}^{m+4}_1$, or it can be written as $\xi=\xi_1-i\xi_2$
where $\{\xi_1,\xi_2\}$ is an orthonormal frame of a spacelike 2-space.

In the first case, such $[\xi]$'s form the projective light cone which could be identified with the sphere $\mathbb{S}^{m+2}$.

In the second case, they form the quadric
\[
Q^{m+2}_+=\{[\xi]|\<\xi,\xi\>=0,\<\xi,\bar\xi\> >0\}
\cong Q^{m+2}\setminus \mathbb{S}^{m+2}
\]
which is a non-compact complex manifold endowed with an indefinite Hermitian metric with signature $(m+1,1)$. In terms of the local lift $\xi$, this Hermitian metric is defined by
\[
h_{\xi}=\frac{1}{\<\xi,\bar\xi\>} \Big\<
d\xi-\frac{\<d\xi,\bar\xi\>}{\<\xi,\bar\xi\>}\xi,
d\bar\xi-\frac{\<d\bar\xi,\xi\>}{\<\xi,\bar\xi\>}\bar\xi
\Big\>.
\]
It is evident that this metric is independent to the choice of the lift $\xi$, and it is invariant under the action of the Lorentz orthogonal group $O(m+3,1)$.\\

A complex curve in $Q^{m+2}_+$ is a holomorphic immersion of a Riemann surface $[\xi]:\overline{M^2}\to Q^{m+2}_+$, given by a local lift $\xi:\overline{M^2}\to \mathbb{C}^{m+4}_1$ satisfying
\[
\frac{\partial}{\partial\bar{z}}\xi=\lambda~\xi \parallel \xi,~~~\text{where}~\frac{\partial}{\partial\bar{z}}
\triangleq\frac{1}{2}(\frac{\partial}{\partial u}+i\frac{\partial}{\partial v})
\]
for some local complex function $\lambda$ and local complex coordinate $z=u+iv$ of $\overline{M^2}$.

A complex curve $[\xi]:\overline{M^2}\to Q^{m+2}_+$ is called \emph{1-isotropic} if and only if the complex differential $d\xi$ is isotropic.

The following is a characterization of such holomorphic 1-isotropic curves. The easy proof is omitted at here.
\begin{lemma}\label{lem-J}
A map $\xi=\xi_1-i\xi_2$ from a Riemann surface $\overline{M^2}$ to $\mathbb{C}^{m+4}_1$ determines a 1-isotropic and holomorphic immersion $[\xi]:\overline{M^2}\to Q^{m+2}_+$ if, and only if,
$\<\xi,\xi\>=0$ and the horizontal part
\[d\xi-\frac{\<d\xi,\bar\xi\>}{\<\xi,\bar\xi\>}\xi\]
is a vector-valued $(1,0)$ form which is isotropic.

In other words, locally there exist $(1,0)$ form $\theta_1+i\theta_2$, 1-form $\theta_{12}$ and orthonormal frame vectors $\xi_1,\xi_2,\eta_1,\eta_2\in \mathbb{R}^{m+4}_1$ on $\overline{M^2}$ such that
\[
d\xi=i\mu(\theta_1+i\theta_2)(\eta_1+i\eta_2)+i\theta_{12}\xi,~~
\mu=\sqrt{\frac{m-1}{4m}}.
\]
\end{lemma}

\section{Mean curvature spheres correspond to a holomorphic 1-isotropic curve}

In this section, we will show that the \emph{mean curvature sphere congruence} of a Wintgen ideal submanifold $M^m\to \mathbb{S}^{m+2}$ is a 2-parameter family of $m$-spheres. Moreover, this Gauss map defines a holomorphic, 1-isotropic curve in $Q^{m+2}_+$, and the submanifold can be recovered from this curve.

\begin{theorem}\label{thm-envelop}
For a Wintgen ideal submanifold $x:M^m\to \mathbb{S}^{m+2}$ which is umbilic-free we have:

(1) The complex vector-valued function $\xi=\xi_{1}-i\xi_{2}$
defines a Gauss map
\[
[\xi]:\overline{M^2}\to Q^{m+2}_+\subset\mathbb{C}P^{m+3}_1.
\]
The image is a 1-isotropic complex curve in the sense that
\begin{equation}\label{eq-isotropic}
\langle d\xi,d\xi\rangle=0,~~
\langle d\xi,d\bar\xi\rangle>0.
\end{equation}

(2) The distribution $\mathbb{D}^{\bot}=\mathrm{Span}\{E_3,\cdots,E_m\}$ is integrable. Its integral submanifolds define a foliation $\mathscr{D}$ of $M^m$.
Moreover, we have the quotient manifold structure
\[
\overline{M^2}=M^m/\mathscr{D}.
\]

(3) The projection $\pi:M^m\to \overline{M^2}$
is a Riemannian submersion (up to the factor $\mu$),
where $\widehat{M}^{m}$ is endowed with the M\"obius metric and $\overline{M^2}$ is endowed with the induced metric from $Q^{m+2}_+$.

(4)The mean curvature spheres $\mathrm{Span}\{\xi_{1},\xi_{2}\}$ is a 2-parameter family.
They envelope a $m$-dimensional submanifold $\widehat{M}^{m}\supset M^{m}$ (it might degenerate at some points).

(5)
$\mathscr{D}$ extends to the whole envelope $\widehat{M}^{m}$
as a foliation by a 2-parameter family of $(m-2)$ dimensional spheres.
In other words, $\widehat{M}^{m}$ can be viewed as a sphere bundle
over the Riemann surface $\overline{M^2}$.
\end{theorem}

\begin{proof}
By the assumption that $M^{m}$ is Wintgen ideal with codimension two, we have obtained the formula \eqref{J1}:
\[
d(\xi_1-i\xi_2)=i\mu(\o_1+i\o_2)(\eta_1+i\eta_2)
+i\theta_{12}(\xi_1-i\xi_2),
\]
with $\eta_1=Y_1+\frac{C_2^1}{\mu}Y,\eta_2=Y_2+\frac{C^1_1}{\mu}Y.$
It follows that the tangent map of $[\xi]$ maps the tangent space $T_pM^m$ at one point $p\in M^m$ to a 1-dimensional complex line in the tangent space of $TQ^{m+2}_+$ at the corresponding image point. Thus the image of $M^m$ under this map is a complex curve of $Q^{m+2}_+$. Intrinsically this is a Riemann surface, which we denote as $\overline{M^2}$. This proves the first part of the conclusion (1).

It is clear from \eqref{eq-omega1a} and \eqref{eq-omega2a}
that the distribution
\[
\mathbb{D}^{\bot}=\mathrm{Span}\{E_3,\cdots,E_m\}
\]
is integrable. These $(m-2)$-dimensional integral submanifolds of $\mathbb{D}^{\bot}$ defines a foliation $\mathscr{D}$ of $M$.
Along each leave of $\mathscr{D}$, the restriction of the tangent map $d\xi$ is parallel to $\xi$ by \eqref{J1}. So $[\xi]:M\to Q^{m+2}_+$ is always constant when restricted to such a leave. This enables us to define a quotient map
\[
M^m\to M^m/\mathscr{D}\cong \overline{M^2}
\]
where each leave of $\mathscr{D}$ is mapped to a single point. In particular this is a submersion between differentiable manifolds. By \eqref{J1} this is even a Riemannian submersion up to the factor $\mu$.
Thus the conclusion (2) and (3) are established.

With respect to the induced Riemann surface structure and local complex coordinate $z$, $dz$ should be a multiple of $\o_1+i\o_2$. Regard $d\xi$ as a vector-valued complex differential form, it follows from \eqref{J1} that $\langle d\xi,d\xi\rangle=0,
\langle d\xi,d\bar\xi\rangle>0.$ So the conclusion (1) has been proved completely.

By the conclusion (1), the mean curvature sphere congruence $\mathrm{Span}\{\xi_1,\xi_2\}$ is obviously a real 2-parameter family. In M\"obius geometry, it is well-known that such a sphere congruence has an envelope if and only if $\xi_1,\xi_2,d\xi_1,d\xi_2$ always span a family of 4-dimensional space-like subspaces, and the points on the envelope is given by the light-like directions located in the orthogonal complements of $\mathrm{Span}\{\xi_1,\xi_2,d\xi_1,d\xi_2\}$.
According to \eqref{J1}, this is exactly such a case.
We denote the envelope as $\widehat{M}^m\supset M^m$, which consists of a 2-parameter family of $(m-2)$-dimensional spheres; each $(m-2)$-dimensional sphere corresponds to
the orthogonal complement of $\mathrm{Span}_{\mathbb{R}}\{\xi_1,\xi_2,\eta_1,\eta_2\}_p$
at one point $p\in M^{m}$. This establishes the conclusion (4).

We assert that every integral submanifold of $\mathbb{D}^{\bot}$ in $M^{m}$ is contained in such a $(m-2)$-dimensional sphere. To see that, take exterior differentiation at both sides of \eqref{J1}. The result looks like
\begin{equation}\label{eta}
d(\eta_1+i\eta_2)=(\o_1+i\o_2)\cdot \eta + (\cdots)(\eta_1+i\eta_2)+(\cdots)(\xi_1-i\xi_2),
\end{equation}
where the component $\eta$ is orthogonal to $\mathrm{Span}_{\mathbb{C}}\{\xi_1,\xi_2,d\xi_1,d\xi_2\}$.
It follows that the subspace
$V=\mathrm{Span}_{\mathbb{R}}\{\xi_1,\xi_2,\eta_1,\eta_2\}$
is invariant along any leave of the foliation $\mathscr{D}$. In particular, the integration of $Y$ along $\mathscr{D}$ is always located in the orthogonal complement of $V$, which implies that any integral submanifold is located on the corresponding $(m-2)$-dimensional sphere. So we obtain the conclusion (5).

In particular, this shows that the foliation structure $\mathscr{D}$ of $M^m$ is indeed a $(m-2)$ dimensional sphere bundle over a Riemann surface $\overline{M^2}$.
This finishes the proof.
\end{proof}

In the statement of the theorem above, we can add that the envelope $\widehat{M}^{m}\supset M^{m}$
is still a Wintgen ideal submanifold (on the subset where it is an immersed submanifold). This is the corollary of the next theorem, which is the converse of Theorem~\ref{thm-envelop}.

\begin{theorem}\label{thm-converse}
Given a holomorphic, isotropic curve $[\xi]:\overline{M}\to \mathbb{Q}$.
The envelope $\widehat{M}^m$ of the corresponding 2-parameter family spheres
is a $m$-dimensional Wintgen ideal submanifold when it is immersed, and $\widehat{M}^m$ has $[\xi]$ as its mean curvature sphere.
\end{theorem}
\begin{proof}
The proof is a little bit long, so a sketch might be helpful.
First we will give a local description of the envelope $\widehat{M}^m$ as an immersion $\hat{Y}:U\times S^{m-2}\to\mathbb{S}^{m+2}$ for $U\subset \overline{M}$. After that we will introduce a moving frame along $U\times S^{m-2}$ and write out the structure equations. The crucial step is to show that $\hat{Y}$ still has $[\xi]$ as its mean curvature sphere.
Then it is straightforward to see that $\hat{Y}$ is Wintgen ideal.

By Lemma~\ref{lem-J}, the assumption of being holomorphic and isotropic implies
\begin{equation}\label{J2}
d(\xi_{1}-i\xi_{2})=i\mu(\theta_1+i\theta_2)(\eta_1+i\eta_2)
+i\theta_{12}(\xi_1-i\xi_2),
\end{equation}
where $\mu=\sqrt{\frac{m-1}{4m}}$, $\theta_1,\theta_2,\theta_{12}$ are real 1-forms locally defined  on the underlying Riemann surface $\overline{M}$.

It follows that $\{\xi_1,\xi_2,d\xi_1,d\xi_2\}$ span a 4-dimensional spacelike subspace $V$. So the sphere congruence $\mathrm{Span}_{\mathbb{R}}\{\xi_1,\xi_2\}$ has an
envelope $\widehat{M}$ which consists of the light-like directions in the orthogonal complement $V^{\bot}$.
Locally we can restrict to a small neighborhood $U\subset\overline{M^2}$ and choose smoothly local pseudo-orthonormal frames at every point $q\in U$
\[
e_0(q),e_1(q),\cdots,e_{m-1}(q)\in V^{\bot}(q), ~~\langle e_0(q),e_0(q)\rangle=-1.
\]
Then one may parameterize $\widehat{M}^{m}$ explicitly as $U\times S^{m-2}\to
\mathbb{S}^{m+2}$ given by
\[
(q,\Theta)~\to~
[\hat{Y}]=[e_0(q)+\sum_{j=1}^{m-1}\Theta_j e_j(q)]
\]
for $q\in U$ and $\Theta=(\Theta_1,\cdots,\Theta_{m-1})\in S^{m-2}$ the coordinates of a unit sphere in $(m-1)$-dimensional Euclidean space.
We want to show that this is a Wintgen ideal submanifold if it is immersed.

We introduce a moving frame along $U\times S^{m-2}$:
\[
\{\hat{Y},Y,\eta_3,\cdots,\eta_m\} \bot \{\eta_1,\eta_2,\xi_1,\xi_2\};
\]
it is required to be orthonormal except that
\[
\langle Y,Y\rangle=0=\langle {\hat Y},{\hat Y}\rangle,~
\langle Y,{\hat Y}\rangle=1.
\]
We emphasize that frame vectors $\{\eta_1,\eta_2,\xi_1,\xi_2\}$ are now defined on $U\times S^{m-2}$ by extending them as constants along the fiber $S^{m-2}$. We will also abuse the notation $\theta_1,\theta_2,\theta_{12}$ to mean their pull-back to $U\times S^{m-2}\subset \widehat{M}$ under the natural projection map. Then \eqref{J2} is still valid under this understanding. Using this moving frame, we write down the structure equations (the convention on the range of indices is $1\le j,k \le m, 3\le a,b\le m$):
\begin{align}
d\xi_{1}&=-\mu\theta_2\eta_1-\mu\theta_1\eta_2+\theta_{12}\xi_2,
\label{3.3}\\
d\xi_{2}&=-\mu \theta_1\eta_1+\mu\theta_2\eta_2-\theta_{12}\xi_1,\label{3.4}\\
d\eta_1&=-{\hat\o}_1Y-\o_1{\hat
Y}+\sum\nolimits_k\Omega_{1k}\eta_k+\mu\theta_2\xi_1+\mu\theta_1\xi_2,\label{3.5}\\
d\eta_2&=-{\hat\o}_2Y-\o_2{\hat
Y}+\sum\nolimits_k\Omega_{2k}\eta_k+\mu\theta_1\xi_1-\mu\theta_2\xi_2,\label{3.6}\\
d\eta_a&=-{\hat\o}_aY-\o_a{\hat Y}+\sum\nolimits_k\Omega_{ak}\eta_k,\label{3.7}\\
dY&=\o Y+\o_1\eta_1+\o_2\eta_2+\sum\nolimits_a\o_a\eta_a,\label{3.9}\\
d{\hat Y}&=-\o {\hat Y}+{\hat\o}_1\eta_1+{\hat\o}_2\eta_2+\sum\nolimits_a{\hat\o}_a\eta_a. \label{3.10}
\end{align}
Here $\o,\o_k,\hat\o_k,\Omega_{jk}$ are
1-forms locally defined on $U\times S^{m-2}\subset\widehat{M}$ which we don't need to know explicitly.

On the other hand, the coefficients of $\xi_1,\xi_2$ in these equations are explicitly determined by \eqref{3.3}, \eqref{3.4} and the orthogonality of the frames.

The crucial observation is that there exist some functions $\hat{F},\hat{G}$ such that
\begin{equation}\label{eq-FG}
{\hat\o_1}=\hat{F}\theta_1+\hat{G}\theta_2,~~ {\hat\o_2}=-\hat{G}\theta_1+\hat{F}\theta_2.
\end{equation}
This follows from differentiating \eqref{3.10} and comparing the coefficients of $\xi_1,\xi_2$; or equivalently, by comparison between \eqref{eta} and \eqref{3.5},\eqref{3.6}. In particular,
\begin{equation}\label{J3}
\hat\o_1+i\hat\o_2=(\hat{F}-i\hat{G})(\theta_1+i\theta_2).
\end{equation}
Now we turn to the key observation as below.\\

\textbf{Claim}:
The submanifold $[\hat{Y}]$ has $\mathrm{Span}_{\mathbb{R}}\{\xi_1,\xi_2\}$ as its mean curvature sphere.\\

To show this, under the induced metric $\langle d\hat{Y},d\hat{Y}\rangle=\sum_{j=1}^m\hat\o_j^2$
we take the orthonormal dual frame $\{\hat{E}_j\}_{j=1}^m$.
Then one can compute the Laplacian $\hat{\Delta}\hat{Y}$ so that we can determine the normal frames of $\hat{Y}$. Because the mean curvature sphere is determined by the subspace
\[
\mathrm{Span}_{\mathbb{R}}\{\hat{Y},\hat{Y}_j,\sum_{j=1}^m \hat{E}_j\hat{E}_j(\hat{Y})\}
=\mathrm{Span}_{\mathbb{R}}\{\hat{Y},\hat{Y}_j, \hat{\Delta}\hat{Y}\},
\]
it suffices to show $\langle\sum_{j=1}^m \hat{E}_j\hat{E}_j(\hat{Y}),\xi_1-i\xi_2\rangle=0,$
or equivalently,
\[
\langle \hat{Y},\sum_{j=1}^m \hat{E}_j\hat{E}_j(\xi_1-i\xi_2)\rangle=0.
\]
This is because $\langle\hat{Y},\xi_r\rangle=0=\langle d\hat{Y},\xi_r\rangle=\langle \hat{Y},d\xi_r\rangle$.
By \eqref{J2} and \eqref{J3}, up to some components orthogonal to $\hat{Y}$, we have equalities
\[
\sum_{j=1}^m\hat{E}_j\hat{E}_j(\xi_1-i\xi_2)
\thickapprox \sum_{j=1}^2\hat{E}_j\hat{E}_j(\xi_1-i\xi_2)
\thickapprox (\hat{E}_1-i\hat{E}_2)(\hat{E}_1+i\hat{E}_2)(\xi_1-i\xi_2)
\thickapprox 0.
\]
This completes the proof of the previous claim. Note that this is quite similar to the proof to Theorem~3.3 in \cite{XLMW}, or even simpler.

Finally, for $\hat{Y}$ we take its canonical lift, whose derivatives are clear to be combinations of $\hat{Y},\eta_1,\eta_2,\eta_a$. Its normal frame is just $\{\xi_1,\xi_2\}$ as we have shown. One reads from \eqref{3.3} and \eqref{3.4} that its M\"obius second fundamental form still take the same form as \eqref{3.1}. Thus it is a Wintgen ideal submanifold.
\end{proof}

\begin{remark}
In the proof above, one can choose a scaling of $\hat{Y}$ suitably so that $\hat{F}^2+\hat{G}^2=1$. Then one can verify that this $\hat{Y}$ is exactly the canonical lift.
It is straightforward to check that the projection
\[
\mathbb{S}^{m+2}\ni[\hat{Y}](p)~~\mapsto~~
\mathrm{Span}_{\mathbb{R}}\{\xi_1,\xi_2\}\in Gr(2,\mathbb{R}^{m+4}_1)
\]
is a Riemannian submersion (up to the factor $\mu=\sqrt{\frac{m-1}{4m}}$) from $\widehat{M}^m$ to the Riemann surface $\overline{M^2}$. This agrees with the conclusion (3) of Theorem~\ref{thm-converse}.
\end{remark}

\section{Relationship with minimal surfaces in $\mathbb{R}^{m+2}$ }

Dajczer and Tojeiro \cite{Dajczer2,Dajczer3} described another construction of almost all codimension two Wintgen ideal submanifolds via minimal surfaces in $\mathbb{R}^{m+2}$.
The following theorem gives a nice geometric correspondence from Wintgen ideal submanifolds to Euclidean minimal surfaces. This result was obtained explicitly in \cite{Rouxel, Dajczer2} and implicitly contained in the main theorem of \cite{Dajczer3}.

\begin{theorem}\label{thm-minimal}
Given a Wintgen ideal submanifold $x:M^m\to\mathbb{R}^{m+2}$.
Then the centers of its mean curvature sphere congruence form a two dimensional submanifold immersed in $\mathbb{R}^{m+2}$ which is a Euclidean minimal surface.
\end{theorem}

\begin{remark}\label{rem-minimal}
It is noteworthy that the statement of Theorem~\ref{thm-minimal} involves some kind of \emph{symmetry breaking}. The Wintgen ideal property and the mean curvature spheres are M\"obius invariant. On the other hand, the centers of those spheres as well as the minimal property depends on a choice of the ambient space metric.

To clarify this problem, consider a given Wintgen ideal submanifold $x:M^m\to\mathbb{S}^{m+2}$.
Assign an arbitrary point $p\in \mathbb{S}^{m+2}\subset \mathbb{R}^{m+3}$ as the north pole, $p=(1,0,\cdots,0)$ up to a choice of the coordinate system. Then take the stereographic projection
\begin{equation}\label{eq-stereo}
\mathbb{R}^{m+3}\supset\mathbb{S}^{m+2}\setminus\{p\}\ni x=(x',\vec{x}'')~\to~\frac{\vec{x}''}{1-x'}\in \mathbb{R}^{m+2},~~~~~x'\in \mathbb{R},~\vec{x}''\in \mathbb{R}^{m+2}.
\end{equation}
Since this is a conformal diffeomorphism, the image is still a Wintgen ideal submanifold, and the mean curvature spheres are mapped to mean curvature spheres. These $m$-spheres is a 2-parameter family according to Theorem~\ref{thm-envelop}.
Under this circumstance, Theorem~\ref{thm-minimal} is equivalent to saying that the centers of these spheres constitute a minimal surface in this ambient flat space. Moreover, no matter which $p\in \mathbb{S}^{m+2}$ is chosen to be $\infty$ (the point at infinity of an affine space $\mathbb{V}_p^{m+2}$), the corresponding locus of the centers of those mean curvature spheres is always a minimal surface in this $\mathbb{V}_p^{m+2}$.

Thus the striking feature of Theorem~\ref{thm-minimal} is that it describes a beautiful property under a \emph{symmetry breaking}, and by the reason of \emph{symmetry}, any of these ways to break symmetry yields the same result.
\end{remark}

Below we provide a new proof to Theorem~\ref{thm-minimal} according to the understanding of Remark~\ref{rem-minimal}.

\begin{proof}[Proof to Theorem~\ref{thm-minimal}]
Assign an arbitrary point $p=[\wp]\in \mathbb{S}^{m+2}$ to be the point at infinity, represented by the light-like vector
\[
\wp=(1,1,\vec{0}),~~~\vec{0}\in\mathbb{R}^{m+2}.
\]
Let $x:M^m\to\mathbb{S}^{m+2}$ be a Wintgen ideal submanifold with Gauss map $[\xi]=[\xi_1-i\xi_2]$. Without loss of generality we may suppose that locally these mean curvature spheres do not pass through $p$, or equivalently, that $\<\xi,\wp\>\ne 0$.
In the Euclidean space
\[\mathbb{V}^{m+2}_p\cong\mathbb{S}^{m+2}\setminus\{p\},
\]
the center of the mean curvature sphere $\mathrm{Span}_{\mathbb{R}}\{\xi_1,\xi_2\}$
is nothing but the inversive image of $p$ with respect to this round sphere. In the light-cone model, the center $[O_{\xi}]$ is the image of $[\wp]$ under the reflection with respect to the subspace $\mathrm{Span}_{\mathbb{R}}\{\xi_1,\xi_2\}$. This is written down explicitly as
\begin{equation}\label{eq-center}
O_{\xi}=\wp-2\<\wp,\xi_1\>\xi_1-2\<\wp,\xi_2\>\xi_2
=\wp-\<\wp,\xi\>\bar\xi-\<\wp,\bar\xi\>\xi,
\end{equation}
where $\bar\xi=\xi_1+i\xi_2$ is the complex conjugation of $\xi$.

To show that this is a minimal surface in $\mathbb{V}^{m+2}_p$, we need to write down the mapping to $\mathbb{R}^{m+2}$ explicitly. We re-write the classical stereographic projection \eqref{eq-stereo} as from the projective lightcone to $\mathbb{R}^{m+2}$:
\[
[\tilde{x}]=\left[\frac{1}{1-x'}(1,x',\vec{x}'')\right]=[1,x]~\to~
\frac{\vec{x}''}{1-x'}.
\]
This amounts to taking a lift of $x$ in the lightcone, denoted as $\tilde{x}$, such that
$\<\tilde{x},\wp\>=-1$, and then projecting $\tilde{x}$ to the orthogonal complement of $\{(1,1,\vec{0}),(1,-1,\vec{0})\}$.

Based on this observation, we need only to take two arbitrary points $p=[\wp],p*=[\wp*]\in\mathbb{S}^{m+2}$, which can always be expressed as
\[
\wp=(1,1,\vec{0}),~\wp^*=(1,-1,\vec{0}),~~~\<\wp,\wp^*\>=-2,
\]
with respect to a suitable Lorentz coordinate system.
The desired local lift $\tilde{x}$ of $O_{\xi}$ in \eqref{eq-center} is given by
\[
\tilde{x}=\frac{1}{\sigma} O_{\xi}=
\frac{1}{\sigma}\big(\wp-\<\wp,\xi\>\bar\xi-\<\wp,\bar\xi\>\xi\big),
\]
\[\sigma\triangleq -\<O_{\xi},\wp\>=-2\<\wp,\xi\>\<\wp,\bar\xi\>.
\]
And the explicit mapping to $\mathbb{R}^{m+2}$ is
\begin{align*}
\tilde{X}&=\tilde{x}+\frac{1}{2}\<\tilde{x},\wp\>\wp^*
+\frac{1}{2}\<\tilde{x},\wp^*\>\wp \\
&=-\frac{1}{2}\wp^*-\frac{1}{\<\wp,\xi\>}\xi
-\frac{1}{\<\wp,\bar\xi\>}\bar\xi
-\frac{\<\wp^*,\xi\>}{4\<\wp,\xi\>}\wp
-\frac{\<\wp^*,\bar\xi\>}{4\<\wp,\bar\xi\>}\wp\\
&=\frac{1}{2}(X+\bar{X}).
\end{align*}
This is exactly the real part of
\begin{equation}\label{eq-X}
X=\frac{-1}{2\langle\xi,\wp\rangle}
\left(2\xi+\langle\xi,\wp\rangle\wp^*
+\langle\xi,\wp^*\rangle\wp\right).
\end{equation}
This $X$ depends on the Gauss map $[\xi]$, which is a mapping from a Riemann surface $\overline{M^2}$ to $Q^{m+2}_+\subset \mathbb{C}P^{m+3}$ by Theorem~\ref{thm-envelop}. So one may regard this is a mapping
$X:\overline{M^2}\to \mathbb{C}^{m+4}_1$. We have the following conclusions:

First, this is indeed a mapping to $\mathbb{C}^{m+2}$ because
$\<X,\wp\>=\<X,\wp^*\>=0.$

Second, this complex vector-valued function is holomorphic. This follows from Theorem~\ref{thm-envelop} that $[\xi]$ is holomorphic, i.e., $\xi_{\bar{z}}=\lambda\xi$ for local coordinate $z$ and local function $\lambda$. Together with \eqref{eq-X}, it implies that $X_{\bar{z}}=0$.

Thirdly, $X_z$ is isotropic. We need only to re-write \eqref{eq-X} as
\[
X=-\frac{1}{2}\wp^*-\tilde\xi-\frac{1}{2}\<\tilde\xi,\wp^*\>\wp,
~~~\tilde\xi\triangleq \frac{1}{\<\xi,\wp\>}\xi.
\]
Since $\<\tilde\xi,\wp\>=1$ is constant, $\<\tilde\xi_z,\wp\>=0$. Moreover, for this codimension two Wintgen ideal submanifold, Theorem~\ref{thm-envelop} already tells us $\xi,\xi_z$ are isotropic. So $\tilde\xi_z$ is isotropic. As the consequence,
$X_z=\tilde\xi_z-(\cdots)\wp$ is isotropic. We also know $|X_z|^2>0$ because $|\xi_z|^2>0$ by Theorem~\ref{thm-envelop}.

From these three conclusions we know $X:\overline{M^2}\to \mathbb{C}^{m+2}$ is an isotropic, holomorphic vector-valued function. So its real part $\tilde{X}$ defines an immersed minimal surface in $\mathbb{R}^{m+2}$.
\end{proof}

In \cite{Dajczer3} the inverse procedure is also given, namely the construction of Wintgen ideal submanifolds of codimension two from Euclidean minimal surfaces.

Instead of repeating their description at here, we will give an interpretation of this relationship between these two classes of geometric objects. By our main results in the previous section, the first class (Wintgen ideal $M^m\to\mathbb{S}^{m+2}$) is essentially the same as holomorphic 1-isotropic curves in $Q^{m+2}_+$. On the other hand, the second class (minimal $\overline{M^2}\to \mathbb{R}^{m+2}$) is well-known to be identical with holomorphic 1-isotropic curves in $\mathbb{R}^{m+2}$. Thus it suffices to establish a correspondence between these two classes of holomorphic 1-isotropic curves.

The proof of Theorem~\ref{thm-minimal} already included an explicit correspondence between $Q^{m+2}$ and $\mathbb{C}^{m+2}\subset \mathbb{C}^{m+4}_1$ as below:
\begin{equation}\label{eq-stereo2}
\pi:[\xi] ~\mapsto~ X
=\frac{-1}{2\langle\xi,\wp\rangle}
\left(\langle\xi,\wp\rangle\wp^*
+\langle\xi,\wp^*\rangle\wp+2\xi\right),
\end{equation}
where we have fixed two lightlike directions $[\wp],[\wp^*]$ satisfying $\langle\wp,\wp^*\rangle=-2$. More precisely, the domain of $\pi$ is an open dense subset of $Q^{m+2}$ where $\langle\xi,\wp\rangle\ne 0$;
the image is the orthogonal complement of $\{\wp,\wp^*\}$.

To find the inverse of $\pi$, put $\wp,\wp^*$ as before and $X=(0,0,X_1,\cdots,X_m)\in\mathbb{C}^{m+2}\subset \mathbb{C}^{m+4}_1$. The inverse $\pi^{-1}$ is then given by
\begin{equation}\label{eq-stereo3}
\pi^{-1}:X \mapsto [\xi],~~~\xi=
\wp^*+\langle X,X\rangle\wp+2X.
\end{equation}
It is easy to verify that the $\xi=\xi_1-i\xi_2$ given above satisfies
\begin{equation}\label{eq-xi}
\<\xi,\xi\>=0,~~\<\xi,\bar\xi\>
=4\<X,\bar{X}\>-2\<\xi,\xi\>-2\<\bar\xi,\bar\xi\>=4|\xi_2|^2.
\end{equation}
So $[\xi]$ defined above is in $Q^{m+2}$ as desired. By the assumption $\<X,\wp\>=\<X,\wp^*\>=0,\<\wp,\wp^*\>=-2$, it is straightforward to show that \eqref{eq-stereo3} and \eqref{eq-stereo2} are inverse mappings to each other.

Indeed \eqref{eq-stereo2} is a complex version of the classical stereographic projection. Take $\wp=(1,1,\vec{0}),~\wp^*=(1,-1,\vec{0})$ and the lift $\xi=(1,\xi',\vec{\xi}'')$. Then \eqref{eq-stereo2} and \eqref{eq-stereo3} read as
\[
\pi:(1,\xi',\vec{\xi}'')~\mapsto~(0,0,\frac{\vec{\xi}''}{1-\xi'}),
\]
\[
\pi^{-1}:(0,0,X)~\mapsto~(|X|^2+1,|X|^2-1,2X).
\]
These formulas are similar to the classical stereographic projection. In particular, when $X\in\mathbb{R}^{m+2}$ we get the old version between the projective lightcone and the Euclidean space.

\begin{theorem}\label{thm-correspondence}
Fix $\wp=(1,1,\vec{0}),\wp^*=(1,-1,\vec{0})\in\mathbb{R}^{m+2}$. Then
the complex stereographic projection in \eqref{eq-stereo2} and its inverse \eqref{eq-stereo3} establish a correspondence between holomorphic 1-isotropic curves in $Q^{m+2}_+\subset\mathbb{C}P^{m+3}_1$ and holomorphic 1-isotropic curves in $\mathbb{C}^{m+2}$. This is a one-to-one correspondence up to the choice of the poles $\wp,\wp^*$.
\end{theorem}
\begin{proof}
Let $[\xi]:\overline{M^2}\to Q^{m+2}_+$ be a holomorphic 1-isotropic curve. The conclusion that $\pi[\xi]=X$ is holomorphic and 1-isotropic is proved almost the same as that of Theorem~\ref{thm-minimal}.
Conversely, given a vector-valued function
\[
X=(0,0,X_1,\cdots,):\overline{M^2}\to\mathbb{C}^{m+2}\subset \mathbb{C}^{m+4}_1,
\]
which is holomorphic and 1-isotropic, i.e., $X_{\bar{z}}=0,\<X_z,X_z\>=0$.
Then $\xi=\wp^*+\langle X,X\rangle\wp+2X$ defined by \eqref{eq-stereo3} obviously satisfy $\xi_{\bar{z}}=0,\<\xi,\xi\>=0$.
From $\xi_z=2\<X_z,X\>\wp+2X_z$ it is 1-isotropic. This finishes the proof.
\end{proof}

\begin{remark}
By Theorem~\ref{thm-minimal} and Theorem~\ref{thm-correspondence}, Wintgen ideal submanifolds of codimension two are constructed from two equivalent geometric objects, i.e., holomorphic 1-isotropic curves in $Q^{m+2}_+$ or in $\mathbb{C}^{m+2}$. The first description is given by us. It has the advantage of being invariant under the M\"obius transformations. Combined with Thereom~\ref{thm-envelop} and Theorem~\ref{thm-converse}, it captures the global structure of a Wintgen ideal submanifold of codimension two. From another point of view \cite{Dajczer3}, minimal surfaces in $\mathbb{R}^{m+2}$ and holomorphic curves in $\mathbb{C}^{m+2}$ are easy to describe explicitly, which would enable us to construct examples of Wintgen ideal submanifolds efficiently.
\end{remark}

\vspace{5mm} \noindent Tongzhu Li,
{\small\it Department of Mathematics, Beijing Institute of
Technology, Beijing 100081, People's Republic of China.
e-mail:{\sf litz@bit.edu.cn}}

\vspace{5mm} \noindent Xiang Ma,
{\small\it School of Mathematical Sciences, Peking University,
Beijing 100871, People's Republic of China.
e-mail: {\sf maxiang@math.pku.edu.cn}}

\vspace{5mm} \noindent Changping Wang
{\small\it School of Mathematics and Computer Science,
Fujian Normal University, Fuzhou 350108, People's Republic of China.
e-mail: {\sf cpwang@fjnu.edu.cn}}

\vspace{5mm} \noindent Zhenxiao Xie,
{\small\it School of Mathematical Sciences, Peking University,
Beijing 100871, People's Republic of China.
e-mail: {\sf xiezhenxiao@126.com}}

\end{document}